\documentclass[12pt, a4paper]{amsart}
\usepackage[hmargin=30mm, vmargin=25mm, includefoot, twoside]{geometry}
\usepackage[bookmarksopen=true]{hyperref}

\usepackage[alphabetic]{amsrefs}

\usepackage{amsfonts,amssymb,verbatim}
\usepackage{latexsym}
\usepackage{mathrsfs}
\usepackage{stmaryrd}
\usepackage{xspace}
\usepackage{enumerate, paralist}
\usepackage{graphicx}
\usepackage[all,cmtip]{xy}
\usepackage[all]{xypic}

\usepackage[usenames,dvipsnames]{color}
 \usepackage{txfonts, pxfonts}

\newtheorem{te}{Theorem}[section]
\newtheorem{cor}[te]{Corollary}

\newtheorem{p}[te]{Proposition}
\theoremstyle{definition}
\newtheorem{de}[te]{Definition}

\theoremstyle{remark}
\newtheorem{rem}[te]{Remark}

\newtheorem{ex}[te]{Example}

\numberwithin{equation}{section}

\newcommand{\Z}{\mathbb{Z}}

\newcommand{\nz}{\mathbb{N}}

\newcommand{\Ker}{\text{Ker}\,}

\newcommand{\vp}{\varphi}
\newcommand{\ve}{\varepsilon}

\begin{document}
\title{Constraint metric approximations and equations in groups}

\author{Goulnara Arzhantseva}
\address[G. Arzhantseva]{Universit\"at Wien, Fakult\"at f\"ur Mathematik\\
Oskar-Morgenstern-Platz 1, 1090 Wien, Austria.}
\email{goulnara.arzhantseva@univie.ac.at}

\author{Liviu P\u aunescu}
\address[L. P\u aunescu]{Institute of Mathematics of the Romanian Academy, 21 Calea Grivitei Street, 010702 Bucharest, Romania}
\email{liviu.paunescu@imar.ro}
\date{}

\thanks{G.A.\ was partially supported by her ERC grant ANALYTIC no.\ 259527, and by the Simons Foundation and by EPSRC grant no EP/K032208/1, while visiting the Isaac Newton Institute (Cambridge). L.P. was supported by grant number PN-II-RU-TE-2014-4-0669 of the Romanian National Authority for Scientific Research, CNCS - UEFISCDI} 
\baselineskip=16pt

\begin{abstract}
We introduce notions of a constraint metric approximation and of a constraint stability of a metric approximation. This is done in the language 
of group equations with coefficients. We give an example of a group which is not constraintly sofic. 
In building it, we find a sofic representation of free group with trivial commutant among extreme points of the convex structure on the space of sofic representations.

We consider the centralizer equation in permutations as an instance of this new general setting.
We characterize permutations $p\in S_k$ whose centralizer is stable in permutations with respect to the normalized Hamming distance, that is,  
a permutation which almost centralizes $p$ is near a centralizing permutation. This answers a question of
Gorenstein-Sandler-Mills~(1962). 

\end{abstract}
\maketitle

\section{Introduction}

The concept of an equation is fundamental in mathematics.
Usual attributes of an equation are variables and coefficients, and one
searches for a solution subject to some admissibility condition or constraint.   
In group theory,  the literature on equations with solutions in (and over) groups is immense.

We deal with \emph{almost} solutions of equations in groups, where `almost' is expressed in terms of a chosen bi-invariant metric on a group. 
This is very much relevant to the currently fast-developing area of
metric approximations of groups. Notable examples of metrically approximable groups are
sofic and hyperlinear groups. Sofic groups  are groups approximable by $(S_k, d_{H})_{k\in \mathbb N}$, symmetric groups of finite degree endowed with the normalized Hamming distance $d_H$.
They have been introduced by Gromov in the context of Gottschalk's surjunctivity conjecture in symbolic dynamics and have been named `sofic' by Weiss.
Hyperlinear groups are groups approximable by $(U(k), d_{HS})_{k\in \mathbb N}$, unitary groups of finite rank endowed with the normalized Hilbert-Schmidt distance $d_{HS}$.
They have been appeared in relation to Connes' embedding conjecture in operator algebras and  have been named `hyperlinear' by R\u adulescu.

The \emph{existence} of a metric approximation of a group $G$ can be formulated through the existence of almost solutions of equations (without coefficients),
defined in the approximating groups by the relator words of $G$. On the other hand, the \emph{stability} of a metric approximation requires that
such an `almost' solution is always close, in the given metric, to an exact solution. Combined, the existence and the stability of a metric approximation of $G$
imply a \emph{robustness} of $G$: the group is forced to be fully residually-$\{$ the class of the approximating groups $\}$.
This opens a possibility to establish the non-existence of certain metric approximations by providing an example of a stable not fully residually-$\{$~the class of the approximating groups~$\}$ group, see our previous work~\cite{ArPau}, initiating the study of stability of metric approximations of arbitrary groups.

In this paper, we introduce \emph{constraint} metric approximations and we analyze questions on their existence and stability. 
A constraint metric approximation is defined
by means of equations with coefficients and such that prescribed constraints are imposed on the almost solutions.
Besides being natural, this generalization  is strongly motivated by several open questions on metric approximations (without constraints)
of famous groups such as Higman group, Burger-Mozes groups, etc.  as well as of wide classes of groups such as one-relator groups. 
Namely, these groups are obtained
via fundamental group-theoretical constructions:
free amalgamated products and HNN-extensions of approximable groups, and our constraint metric approximations provide
a rigorous framework to analyze metric approximability of such constructions.

The paper is organized as follows. 

In Section~\ref{sec:eq},  we formalise the new concepts of constraint metric approximations and
constraint stability in the most general setting: the approximating family, its cardinality, and the associated bi-invariant metrics are arbitrary;
the approximated group is arbitrary as well (e.g. it is not required to be finitely presented or finitely generated).
We prove both ultraproduct and algebraic characterizations of existence and stability of such metric approximations.
These encompass and generalize all such prior characterizations. 

In Section~\ref{sec:sofic}, we focus on constraint sofic approximations.
We give  examples of sofic groups admitting no constraint sofic approximations, with prescribed constraints.
In building such examples, we produce a sofic representation of a non-abelian free group that acts ergodically on the Loeb measure space and such that its commutant is trivial. This solves in the negative an open question from~\cite{Pa1}. 

Section~\ref{sec:centra} is devoted to the constraint stability in permutations of the commutator equation or, in other words,
to the `stability of centralizer' in permutations. Such a stability was first considered in \cite{GSM} (see also~\cite{M}), where a quantitative answer was obtained
for permutations with a specific cycle structure. In Theorem~\ref{thm:centra}, answering a question of Gorenstein-Sandler-Mills~(1962),
we give a qualitative characterization of stable centralizers in permutations in full generality.

\section{Metric approximations and equations in groups}\label{sec:eq}

\subsection{Equations in groups}

Let $\mathbb F_{m}$ denote the free group of rank $m$  freely generated by $\bar x_1, \ldots, \bar x_m$, and $\mathbb F_\ell$ denote the free group of rank $\ell$ freely generated by $\bar a_1,\ldots, \bar a_\ell$. We endow $\mathbb F_\ell\ast \mathbb F_m$ with the word length metric induced by the generating set $\bar a_1,\ldots, \bar a_\ell, \bar x_1,\ldots, \bar x_m.$

A word $w$ representing an element in the free product $\mathbb F_\ell\ast \mathbb F_m$ 
defines an \emph{equation} in $m$ \emph{variables $\bar x_1, \ldots, \bar x_m$} with $\ell $ \emph{coefficients $\bar a_1, \ldots, \bar a_\ell$} for some $\ell\geqslant 0$:
\begin{equation}\label{eq:main}
w(\bar a_1,\ldots, \bar a_\ell, \bar x_1,\ldots, \bar x_m)=1.
\end{equation}

Let $W\subseteq \mathbb F_\ell\ast \mathbb F_m$ be a finite subset, $\ell\geqslant 0, m\geqslant 1$. 
Let $H$ be a group equipped with a bi-invariant distance $d_H$, and  $\ve_H>0$ be a fixed real number.

\begin{de}[Solution and almost solution]\label{def:sol}
Elements $h_1, \ldots, h_m\in H$ are a \emph{solution} of $W$ in $H$ with coefficients $a_1, \ldots, a_\ell\in H$ if
\begin{equation}\label{eq:sys}
w(a_1, \ldots, a_\ell, h_1, \ldots, h_m)=1_H, \quad \forall w\in W,
\end{equation}
 where $1_H$ denotes the identity of $H$.

Elements $h_1, \ldots, h_m\in H$ are a \emph{strong solution} of $W$ in $H$ with coefficients $a_1, \ldots, a_\ell\in H$ if
\begin{align}
w(a_1, \ldots, a_\ell, h_1, \ldots, h_m)= 1_H, \quad  & \forall w\in W;\label{eq:strongs}\\
d_H(w(a_1, \ldots, a_\ell, h_1, \ldots, h_m),1_H)\geqslant \ve_H, \quad  & \forall w\not\in \langle W\rangle,\label{eq:stronginj}
\end{align}
where $\langle W\rangle$ denotes the normal subgroup of $\mathbb F_\ell\ast \mathbb F_m$ generated by $W.$

Elements $h_1, \ldots, h_m\in H$ are a \emph{$\delta$-solution} of $W$ in $H$ with coefficients $a_1, \ldots, a_\ell\in H$, for some $\delta>0$, if
\begin{equation}\label{eq:almost}
d_H\left(w(a_1, \ldots, a_\ell, h_1, \ldots, h_m),1_{H}\right)<\delta,  \quad \forall w\in W.
\end{equation}

Elements $h_1, \ldots, h_m\in H$ are a \emph{strong $\delta$-solution} of $W$ in $H$ with coefficients $a_1, \ldots, a_\ell\in H$, for some $\delta>0$, if $\forall w\in \mathbb F_\ell\ast\mathbb F_m$ of length at most $1/\delta$ we have
\begin{align}
d_H\left(w(a_1, \ldots, a_\ell, h_1, \ldots, h_m),1_{H}\right)<\delta,\quad & \forall w\in W;\label{eq:salmost}\\
d_H\left(w(a_1, \ldots, a_\ell, h_1, \ldots, h_m),1_{H}\right)\geqslant \ve_H-\delta, \quad & \forall w\not\in \langle W\rangle.\label{eq:sinj}
\end{align}

We say that \emph{$W$ is solvable} (resp. \emph{strongly solvable}, \emph{$\delta$-solvable}, and \emph{strongly $\delta$-solvable} )  in $H$, with coefficients $a_1, \ldots, a_\ell\in H$, if there exists
a solution (resp. strong solution, $\delta$-solution, and strong $\delta$-solution) of $W$ in $H$, with coefficients $a_1, \ldots, a_\ell\in H$.
\end{de}

We collect now standard interpretations of solvability, \emph{in} the group,
of equations with coefficients. This is inspired by the theory of algebraic extensions of a field and goes back to B.H.~Neumann~\cite{N}.
We add arguments for completeness. 

Given $\ell\geqslant 0$ elements $a_1, \ldots, a_\ell\in H,$ there is a unique group homomorphism $\jmath=\jmath(a_1,\ldots, a_\ell)\colon \mathbb F_\ell\ast\mathbb F_m\to H\ast \mathbb F_m$ such that $\bar a_i\mapsto a_i, \bar x_i\mapsto \bar x_i.$ For a finite subset $W\subseteq \mathbb F_\ell\ast \mathbb F_m$, we denote by $V=\jmath (W)$ its image and
by $\langle V\rangle$ the normal subgroup generated by $V$. Thus, elements of $V$ are words from $W$ with letters $\bar a_i$ are
 substituted by letters $a_i$. 

\begin{te}\label{thm:sol}
The following are equivalent.
\begin{enumerate}
\item $W$ is solvable in $H$, with coefficients $a_1, \ldots, a_\ell\in H$.
\item The natural homomorphism $\iota\colon H\to H\ast \mathbb F_m/ \langle V\rangle, h\mapsto h \langle V\rangle$ is split-injective, that is, it has a left inverse.
\item 
\begin{enumerate}
\item There is an injective homomorphism $\iota\colon H\hookrightarrow H\ast \mathbb F_m / \langle V\rangle;$
\item There exists a normal subgroup $N\unlhd H\ast \mathbb F_m/ \langle V\rangle$ such that $HN=H\ast \mathbb F_m / \langle V\rangle$ and $H\cap N=\{1\};$ that is,
$H\ast \mathbb F_m/ \langle V\rangle=N\rtimes H.$
\end{enumerate}
\item There is a split short exact sequence
$$
1\to N\hookrightarrow H\ast \mathbb F_m/ \langle V\rangle \twoheadrightarrow H\to 1.
$$
\item $H$ is a retract of $H\ast \mathbb F_m/ \langle V\rangle$.
\end{enumerate}

\end{te}
\begin{proof}
$(1)\Leftrightarrow (2).$ If $W$ is solvable in $H$, with coefficients $a_1, \ldots, a_\ell\in H$ and $h_1,\ldots, h_m\in H$ are a fixed solution, then
we define a map $$\eta\colon H\ast \mathbb F_m/ \langle V\rangle\twoheadrightarrow H,$$ 
by assigning $\eta(\bar x_1 \langle V\rangle)=h_1, \ldots, \eta(\bar x_m \langle V\rangle)=h_m, \eta(h \langle V\rangle)=h$ for all $h\in H$, and extending
this multiplicatively to all $u \langle V\rangle\in H\ast \mathbb F_m/ \langle V\rangle,$ where $u$ is an arbitrary word representing an element of $H\ast \mathbb F_m$. This map is well-defined: if $u'\in u \langle V\rangle$, then $H\ast \mathbb F_m\ni u^{-1}u'$ is a product of conjugates of $v^{\pm 1}$ for some finitely many $v\in V$, and hence, evaluates to $1_H$ under $\eta$ since $h_1,\ldots,h_m$ are a solution.
It is obvious that $\eta$ is a homomorphism which is a left inverse of $\iota$.

Conversely, if $\eta\colon H\ast \mathbb F_m/ \langle V\rangle \to H$ is a left inverse of $\iota$, then we set $h_1=\eta( \bar x_1 \langle V\rangle), \ldots, 
h_m=\eta(\bar x_m \langle V\rangle)$. Using that $\eta(\iota(h))=h$ for all $h\in H$ and $\eta$ is a homomorphism, it is immediate that these $h_1, \ldots, h_m$ are 
a solution of $W$.

$(2)\Leftrightarrow (3).$ A split-injective map is always injective. So, if $\iota$ is split-injective, we set $N=\Ker\eta$, the kernel of a left inverse $\eta$  of $\iota$. 
Then, denoting $x=u\langle V\rangle$ and using that $\eta(\iota(h))=h$ for all $h\in H$, we have  
$$\eta(x\cdot \iota(\eta(x^{-1})))=\eta(x)\eta(\iota(\eta(x^{-1})))=\eta(x)\eta(x^{-1})=1_{H\ast \mathbb F_m/ \langle V\rangle}.$$
Hence, $x\cdot \iota(\eta(x^{-1}))\in N$ and $x=\iota(\eta(x))\cdot n$ for some $n\in N,$ that is, $HN=H\ast \mathbb F_m / \langle V\rangle.$
If $\iota(\eta(x))$ belongs to $N$, then $\eta(\iota(\eta(x)))= \eta(x)=1_{H}$ and hence $\iota(\eta(x))=1_{H\ast \mathbb F_m/ \langle V\rangle},$ whence
$H\cap N=\{1\}.$ 

Conversely, a left inverse of $\iota$ is given by the natural map $HN\ni (h,n)\mapsto h\in H.$ 

$(3)\Leftrightarrow (4) \Leftrightarrow (5)$ are well-known and can be checked by definitions of the involved concepts.
\end{proof}

Using basic terminology of algebraic geometry, the set of
all solutions of $W\subseteq \mathbb F_\ell\ast \mathbb F_m$ in $H$ with coefficients $a_1, \ldots, a_\ell\in H$ form an \emph{algebraic set} in $H^m$ defined by $W$,
this algebraic set is uniquely (cf. Corollary~\ref{cor:equiv}) defined by its \emph{radical} $\langle V\rangle$, and the quotient 
$H\ast \mathbb F_m/ \langle V \rangle$ is the \emph{coordinate group} of system $W$.

In  a categorical language:   $H\ast \mathbb F_m$ (often denoted by $H[\bar x_1, \ldots, \bar x_m]$ and viewed as
an analogue of a polynomial algebra with $H$ playing a role of the coefficients)
is the free object in the category of \emph{$H$-groups}, i.e. groups containing a designated copy of $H$ viewed up to \emph{$H$-morphisms},
group homomorphisms trivial on those prescribed copies of $H$.  Theorem~\ref{thm:sol}
implies that each quotient $H\ast \mathbb F_m/ \langle V \rangle$ is an object in this category and every solution of $W$
can be described as an $H$-morphism $\eta\colon H\ast \mathbb F_m/ \langle V \rangle \twoheadrightarrow H.$

\subsection{Constraint approximation and constraint stability of systems}

Let $\mathcal{F}=\left(
G_{\alpha }, d_{\alpha}, \varepsilon _{\alpha }\right) _{\alpha \in I}$ be an \emph{approximating family}:
$G_{\alpha }$ is a  group with a bi-invariant distance $d_{\alpha
} $ and identity element $1_{\alpha }$,  and $%
\varepsilon _{\alpha }$ is a strictly positive real number such that $%
\epsilon:=\inf_{\alpha  \in I}\varepsilon _{\alpha }>0$. For each $\alpha\in I,$ we fix elements $a_1^\alpha, \ldots, a_\ell^\alpha\in G_\alpha$.

The following definition encompasses both well-known metric approximations such as sofic, hyperlinear, etc. approximations and
metric approximations by wider families (notice the use of the index set $I$ instead of usual $\mathbb N$).

\begin{de}[Constraint $\mathcal F$-approximation]\label{def:capp}
A finite system $W\subseteq\mathbb F_\ell\ast\mathbb F_m$ is \emph{constraint $\mathcal F$-approximable} with respect to $a_1^\alpha, \ldots, a_\ell^\alpha\in G_\alpha, \alpha\in I,$ if $\forall \delta>0\, \exists \alpha\in I$ such that $W$ is strongly $\delta$-solvable in $G_\alpha$ with coefficients $a_1^\alpha,\ldots,a_\ell^\alpha\in G_\alpha$. 

A finite system $W\subseteq \mathbb F_\ell\ast \mathbb F_m$ is \emph{$\mathcal F$-approximable}  if it is constraint $\mathcal F$-approximable with respect to the 
identity coefficients 
$a_1^\alpha=1_\alpha, \ldots, a_\ell^\alpha=1_\alpha\in G_\alpha, \alpha\in I.$

\end{de}

Let us now formalize a possibility when an almost solution is close, in the given metric, to a solution, while
the coefficients are prescribed.

\begin{de}[Constraint $\mathcal F$-stability]\label{def:cstab}

A finite system $W\subseteq \mathbb F_\ell\ast \mathbb F_m$ is \emph{constraint $\mathcal F$-stable}  with respect to  $a_1^\alpha, \ldots, a_\ell^\alpha\in G_\alpha, \alpha\in I,$ if $\forall \varepsilon>0\, \exists \delta>0$ such that $\forall \alpha\in I\, \forall g_1, \ldots, g_m\in G_\alpha$ a $\delta$-solution
of $W$ with coefficients $a_1^\alpha, \ldots, a_\ell^\alpha$, there exist $\tilde g_1, \ldots, \tilde g_m \in G_\alpha$ a solution of $W$ with coefficients $a_1^\alpha, \ldots, a_\ell^\alpha$ such that $d_\alpha(g_i, \tilde g_i)<\varepsilon.$

A finite system $W\subseteq \mathbb F_\ell\ast \mathbb F_m$ is \emph{$\mathcal F$-stable}  if it is constraint $\mathcal F$-stable with respect to the 
identity coefficients 
$a_1^\alpha=1_\alpha, \ldots, a_\ell^\alpha=1_\alpha\in G_\alpha, \alpha\in I.$

\end{de}

The next notion is to take also into account the non-solutions of equations.

\begin{de}[Constraint weak $\mathcal F$-stability]\label{def:cwstab}

A finite system $W\subseteq \mathbb F_\ell\ast\mathbb F_m$ is  \emph{constraint weakly $\mathcal F$-stable}  with respect to $a_1^\alpha, \ldots, a_\ell^\alpha\in G_\alpha, \alpha\in I,$
 if $\forall \varepsilon>0\, \exists \delta>0$ such that $\forall \alpha\in I\, \forall g_1, \ldots, g_m\in G_\alpha$ a strong $\delta$-solution of $W$ with coefficients $a_1^\alpha, \ldots, a_\ell^\alpha$, there exist $\tilde g_1, \ldots, \tilde g_m \in G_\alpha$ a solution of $W$ with coefficients $a_1^\alpha, \ldots, a_\ell^\alpha$ such that $d_\alpha(g_i, \tilde g_i)<\varepsilon.$

A finite system $W\subseteq \mathbb F_\ell\ast \mathbb F_m$ is  \emph{weakly $\mathcal F$-stable}  if it is constraint weakly $\mathcal F$-stable with respect to the 
identity coefficients  $a_1^\alpha=1_\alpha, \ldots, a_\ell^\alpha=1_\alpha\in G_\alpha, \alpha\in I.$

\end{de}

\subsection{Constraint approximation of groups}
Given a finite system $W\subseteq \mathbb F_\ell\ast \mathbb F_m$, let $G=\mathbb F_\ell\ast\mathbb F_m/ \langle W\rangle$ be the quotient of $\mathbb F_\ell\ast\mathbb F_m$ by the normal closure of $W$, and 
$$\rho\colon \mathbb F_\ell\ast\mathbb F_m\twoheadrightarrow  \mathbb F_\ell\ast\mathbb F_m/ \langle W\rangle$$
be the canonical projection.
We denote by $a_i=\rho(\bar a_i)$, for the generators of $\mathbb F_\ell$, and  by $x_i=\rho(\bar x_i)$, for the generators of $\mathbb F_m$.
That is, $G$ is given by the presentation
\begin{equation}\label{eq:G}
G=\langle a_1,\ldots , a_\ell,  x_1, \ldots,  x_m \mid w(a_1, \ldots, a_\ell, x_1, \ldots, x_m)=1, \forall w\in W\rangle
\end{equation}

We would like now to determine the constraint $\mathcal F$-approximability of a system $W$
through properties of the group $G$ and vise versa. For this we use metric ultraproducts of groups $G_\alpha, \, \alpha\in I$. Observe that the index set $I$ is arbitrary and there is no control over $G_\alpha\in\mathcal F$ that are used in Definition~\ref{def:capp}. 
For different values of $n\in\nz^*$ (taking $n=1/\delta$ in Definition~\ref{def:capp}), $\alpha\in I$ can vary greatly or could be the same, in which case we 
are lead to construct the ultrapower of $G_\alpha$. This is why we use a function $f\colon\nz\to I$, together with a non-principal ultrafilter $\omega$ over $%
\nz $. 

The \emph{metric ultraproduct} $\prod_{k\to\omega}\left( G_{f(k)
},d_{f(k) }\right) $ of the family of bi-invariant metric groups $\left(
G_{f(k) },d_{f(k)}\right) _{k \in \nz}$  is the quotient of the direct
product $\prod_{k \in \nz}G_{f(k)}$ by the normal subgroup consisting of
those elements $\left( g_{k }\right)_{k\in \nz} $ such that $\lim_{k
\rightarrow \omega}d_{f(k) }\left( g_{k },1_{f(k) }\right) =0$%
. The metric ultraproduct $\prod_{k\to\omega}\left( G_{f(k) },d_{f(k)}\right) 
$ is endowed with a canonical bi-invariant metric $d_{\omega}$,
obtained as the quotient of the bi-invariant pseudometric on $\prod_{k
\in \nz}G_{f(k) }$, and defined by $d_{\omega}\left( \left( g_{k
}\right)_{k\in \nz}  ,\left( h_{k }\right)_{k\in \nz}  \right) =\lim_{k \rightarrow 
\omega}d_{f(k) }\left( g_{k },h_{k }\right) $.
We write $1_\omega$ for the identity element of this group.

\begin{te}\label{thm:ultra}
A finite system $W\subseteq\mathbb F_\ell\ast\mathbb F_m$ is constraint $\mathcal F$-approximable with respect to $a_1^\alpha, \ldots, a_\ell^\alpha\in G_\alpha, \alpha\in I,$ if and only if there exist a non-principal ultrafilter $\omega$ over $\nz$, a function $f\colon\nz\to I,$ and a group homomorphism $$\theta\colon G\rightarrow \prod_{k\to\omega}\left( G_{f(k)},d_{f(k)}\right), \hbox{such that}$$  
\begin{enumerate}
\item[(i)] $d_{\omega%
}\left( \theta(g) ,\theta(h) \right) \geqslant \epsilon=\lim_{k\to\omega}\varepsilon _{f(k)} $
for every nontrivial distinct $g,h\in G$;
\item[(ii)] $\theta ( a_i)=\prod_{%
k\to\omega}a_i^{f(k)}$, for each $i=1,\ldots,\ell.$
\end{enumerate}
\end{te}
\begin{proof}
Take a sequence $(\delta_k)_{k\in\mathbb N}$ with $\delta_k\in\mathbb{R}_+^*$ and such that $\delta_k\to0$ as $k\to\infty$, and construct $\delta_k$-strong solutions 
$g_1^{f(k)}, \ldots, g_m^{f(k)}\in G_{f(k)}$ of $W$, with respect to $a_1^{f(k)}, \ldots, a_\ell^{f(k)}\in G_{f(k)}$ in the sens of Definition~\ref{def:sol}. Here, each $f(k)=\alpha(\delta_k)$ is given by Definition~\ref{def:capp}.
 Then define $\theta(a_i)=\prod_{%
k\to\omega}a_i^{f(k)}$, for each $i=1,\ldots,\ell$ and $\theta(x_i)=\prod_{%
k\to\omega}g_i^{f(k)}$, for each $i=1,\ldots, m.$  It follows from (\ref{eq:salmost}) that $\theta$ is a homomorphism and
from (\ref{eq:sinj}) that $\theta$ satisfies condition (i). 

For the reverse implication, fix $\delta>0$. Let $\theta_k(x_i)\in G_{f(k)}$ be such that $\theta(x_i)=\Pi_{k\to\omega}\theta_k(x_i)$. There are finitely many inequalities that elements $\theta_k(x_1),\ldots, \theta_k(x_m)$ in $G_{f(k)}$ must obey for them to be a $\delta$-strong solution of $W$. Since all of these inequalities are satisfied in the ultralimit, it follows that there exists $k\in\nz$ such that $\theta_k(x_1),\ldots,\theta_k(x_m)$ are indeed a $\delta$-strong solution of $W$.
\end{proof}

\begin{de}[Equivalent systems]
Two systems $W_1\subseteq\mathbb F_\ell\ast\mathbb F_{m_1}$ and $W_2\subseteq\mathbb F_\ell\ast\mathbb F_{m_2}$ are called \emph{equivalent} if there exists a group isomorphism $\phi\colon \mathbb F_\ell\ast\mathbb F_{m_1}/\langle W_1\rangle\to \mathbb F_\ell\ast\mathbb F_{m_2}/\langle W_2\rangle$ such that $\phi(\bar a_i\langle W_1\rangle)=\bar a_i\langle W_2\rangle$ for each $i=1,\ldots,\ell$, where $\bar a_1,\ldots, \bar a_\ell$ are the generators of $\mathbb F_\ell$.
\end{de}

\begin{cor}\label{cor:equiv ap}
Let $W_1$ and $W_2$ be two equivalent systems. If $W_1$ is constraint $\mathcal F$-approximable then so is $W_2$.
\end{cor}
\begin{proof}
It is straightforward, by Theorem \ref{thm:ultra}.
\end{proof}

\begin{de}[Constraint $\mathcal F$-approximation of groups]\label{def:gapp}
A group $G= \mathbb F_\ell\ast\mathbb F_m/ \langle W\rangle$ is \emph{(constraint) $\mathcal F$-approximable} with respect to $a_1^\alpha, \ldots, a_\ell^\alpha\in G_\alpha, \alpha\in I,$
if $W\subseteq \mathbb F_\ell\ast\mathbb F_m$ is (constraint) $\mathcal F$-approximable with respect to $a_1^\alpha, \ldots, a_\ell^\alpha\in G_\alpha, \alpha\in I,$ in the sens of Definition~\ref{def:capp}.

\end{de}

By Corollary~\ref{cor:equiv ap}, the notion of constraint $\mathcal F$-approximability of a group is well-defined: 
it does not depend on the choice of the group presentation. 

\begin{rem}\label{rem:wlog}
Our definition of (constraint) $\mathcal F$-approximation of a group encompasses all finitely presented groups, considered with chosen finitely many group elements, \emph{coefficients}.
Indeed, let $G$ be given by an arbitrary  finite presentation $\langle x_1, \ldots, x_m \mid r=1,\, \forall r\in R\rangle$ and $a_1, \ldots,  a_\ell\in G$ be fixed $\ell\geqslant 0$ group elements.  The system $W\subseteq \mathbb F_\ell\ast\mathbb F_m$ associated to this data  is defined as follows: add elements $ a_1,\ldots,  a_\ell$ to the generators of $G$ and consider a new, still finite, presentation of $G\simeq\mathbb F_\ell\ast\mathbb F_m/\langle W\rangle$ for $W\subseteq \mathbb F_\ell\ast\mathbb F_m$, where $W=R\sqcup A$ and $A$ consists of words $ a_i=  a_i(  x_1, \ldots,  x_m)$ representing each of the fixed coefficients $ a_i\in G, i=1, \ldots, \ell$ in the initial generators $ x_1, \ldots,  x_m$ of $G$.

The notion extends to infinitely presented and countably generated groups.
If $G=\langle x_1, \ldots, x_m \mid r'=1,\, \forall r'\in R'\rangle$,  where the set of relators  $R'\subseteq \mathbb F_m$ is infinite,
we say that $G$ is \emph{(constraint) $\mathcal F$-approximable} if, for every finite subset $R\subseteq R'$, the finitely presented group 
$\langle x_1, \ldots, x_m \mid r=1,\, \forall r\in R\rangle$ is (constraint) $\mathcal F$-approximable.
A countable group $G$ is  said to be \emph{(constraint) $\mathcal F$-approximable}  if so is its every finitely generated subgroup.
\end{rem}

\begin{ex}[Centralizer equation]
Set $\ell=1, m=1,$ and $W=\{axa^{-1}x^{-1}\}\subseteq  \mathbb Z\ast\mathbb Z$. Here $a$ is the coefficient and  $x$ is the variable of the 
centralizer equation $axa^{-1}x^{-1}=1.$
Then, by definition,  $W$  is (constraint) $\mathcal F$-approximable with respect to $a^\alpha\in G_\alpha, \alpha\in I$ if
and only if the group $G=\langle a,x \mid axa^{-1}x^{-1}=1\rangle $ is (constraint) $\mathcal F$-approximable with respect to $a^\alpha\in G_\alpha, \alpha\in I.$
\end{ex}

Theorem~\ref{thm:ultra} allows us to interpret the existence of constraint $\mathcal F$-approximations of a group
via (exact) solutions of group relator equations, with coefficients, in the metric ultraproduct of groups from the approximating family.

\begin{cor}\label{cor:ext}
A group $G= \mathbb F_\ell\ast\mathbb F_m/ \langle W\rangle$ is constraint $\mathcal F$-approximable with respect to $a_1^\alpha, \ldots, a_\ell^\alpha\in G_\alpha, \alpha\in I,$
if and only if there exist a non-principal ultrafilter $\omega$ over $\nz$ and a function $f\colon\nz\to I$ such that the system $W$ is strongly solvable in $\mathbb G=\prod_{k\to\omega}\left( G_{f(k)},d_{f(k)}\right)$, with  $\ell$ coefficients $\prod_{%
k\to\omega}a_i^{f(k)}\in \mathbb G, i=1,\ldots,\ell$, and $\ve_\mathbb G=\lim_{k\to\omega}\ve_{f(k)}$. 
\end{cor}

Thus, deciding whether or not there exists a (constraint) $\mathcal F$-approximation of a group reduces to solving equations (with coefficients) given by the group relators in `big'  groups such as $\mathbb G.$ This approach is rather unexplored, besides Theorem~\ref{thm:sol} in this setting (taking $H=\mathbb G$).

We give now an algebraic characterisation of constraint $\mathcal F$-approximability of a group.

\begin{te}\label{te:alg}
A group $G= \mathbb F_\ell\ast\mathbb F_m/ \langle W\rangle$ 
is constraint $\mathcal F$-approximable with respect to $a_1^\alpha, \ldots, a_\ell^\alpha\in G_\alpha, \alpha\in I,$ if and only if $\forall n\in \mathbb N^*$ there exists a homomorphism $\pi_n\colon \mathbb F_\ell\ast\mathbb F_m\to G_\alpha,$ for some $G_\alpha\in \mathcal F,$ such that 

\begin{enumerate}
\item[(1)] $d_{\alpha }\left( \pi_n \left( r\right), 1_\alpha \right) <1/n$ for every group relator $r\in W\subseteq \mathbb F_\ell\ast\mathbb F_m$ of length at most $n$, 

\item[(2)] $d_{\alpha }\left( \pi_n \left( w\right) , 1_\alpha \right)
>\varepsilon _{\alpha }-1/n$ for every $w\in \mathbb F_\ell\ast\mathbb F_m$ of length at most $n$ with $\rho(w)\not=1_G$, 

\item[(3)] $d_{\alpha }\left( \pi_n \left( \bar a_i\right) , a_i^\alpha \right)
<1/n$ for every generator $\bar a_i$ of  $\mathbb F_\ell$, for $i=1,\ldots,\ell.$
\end{enumerate}

\end{te}
\begin{proof}
Let $G= \mathbb F_\ell\ast\mathbb F_m/ \langle W\rangle$ be a constraint $\mathcal F$-approximable group with respect to $a_1^\alpha, \ldots, a_\ell^\alpha\in G_\alpha, \alpha\in I$.
For each $n\in\nz$ apply Definition~\ref{def:capp} (and, hence, Definition~\ref{def:sol} with $n=1/\delta$) to get a homomorphism $\pi_n\colon \mathbb F_\ell\ast\mathbb F_m\to G_{\alpha}$ with the required properties. In the third condition, we actually have $\pi_n(\bar a_i)=a_i^\alpha$ for $i=1,\ldots,\ell$.

For the reverse implication, we define the homomorphism to the ultraproduct $\pi\colon\mathbb F_\ell\ast\mathbb F_m\to\prod_{n\to\omega}\left( G_{f(n)},d_{f(n)}\right)$ by $\pi(w)=\prod_{n\to\omega}\pi_n(w)$. The first condition implies $\langle W\rangle\subseteq \Ker\pi$. The second condition implies $d_\omega(\pi(w), 1_\omega)\geqslant\epsilon$ for any $w\notin \langle W\rangle$, and the third condition implies that $\pi(\bar a_i)=\prod_{n\to\omega}a_i^{f(n)}$, for $i=1,\ldots,\ell$. We deduce that $\Ker\pi=\langle W\rangle$ and $\pi$ factors to a group homomorphism of $G$ as in Theorem~\ref{thm:ultra}, whence the conclusion.
\end{proof}

\subsection{Constraint stability of groups}
We turn to the stability of metric approximations and characterize (constraint) $\mathcal F$-stability
through the existence of (constraint) lifts. This generalizes our result~\cite[Theorem 4.2 and observation thereafter]{ArPau} to a much wider setting.

\begin{de}[Constraint $\mathcal F$-stability of groups]\label{def:gcstab}
A group $G= \mathbb F_\ell\ast\mathbb F_m/ \langle W\rangle$ is \emph{(constraint) $\mathcal F$-stable} with respect to $a_1^\alpha, \ldots, a_\ell^\alpha\in G_\alpha, \alpha\in I,$
if $W\subseteq \mathbb F_\ell\ast\mathbb F_m$ is (constraint) $\mathcal F$-stable with respect to $a_1^\alpha, \ldots, a_\ell^\alpha\in G_\alpha, \alpha\in I,$ in the sens of Definition~\ref{def:cstab}.
\end{de}

\begin{de}[Constraint lifts]\label{def:perfect} Let $G=\mathbb F_\ell\ast\mathbb F_m/\langle W\rangle.$ 
Given a non-principal ultrafilter $\omega$ over $\nz$ and a function $f\colon\nz\to I,$ 
a (not necessarily injective) group homomorphism $$\theta\colon G\rightarrow \prod_{k\to\omega}\left( G_{f(k)},d_{f(k)}\right), \hbox{ such that } \theta( a_i)= \prod_{k\to\omega}a_i^{f(k)}, \, i=1, \ldots, \ell$$ is called \emph{constraint liftable} 
with respect to $a_1^\alpha, \ldots, a_\ell^\alpha\in G_\alpha, \alpha\in I,$
if for each $k\in\nz$ there exist $g_i^k\in G_{f(k)}, i=1,\ldots, m$ such that 
$g_1^k,\ldots,g_m^k$ are a solution of $W$ in $G_{f(k)}$ with coefficients $a_1^{f(k)},\ldots, a_\ell^{f(k)}\in G_{f(k)}$ and
$\theta(x_i)=\prod_{%
k\to\omega}g_i^{k}$ for each $i=1,\ldots, m$. 

A \emph{constraint lift} of $\theta$ is the homomorphism $\tilde\theta\colon G\rightarrow \prod_{%
k\in\nz}\left( G_{f(k) },d_{f(k) }\right)$, defined by 
$\tilde\theta(a_i)=\prod_{%
k\in\nz}a_i^{f(k)}$ for each $i=1,\ldots, \ell$
and  by $\tilde\theta( x_i)=\prod_{%
k\in\nz}g_i^k$ for each $i=1,\ldots, m.$

Again, we have notions of a \emph{liftable} homomorphism (cf.~\cite[Definition 4.1]{ArPau}, where it was named \emph{perfect}), and of  a \emph{lift}, 
whenever we consider the preceding conditions  with respect to
the  identity coefficients 
$a_1^\alpha=1_\alpha, \ldots, a_\ell^\alpha=1_\alpha\in G_\alpha, \alpha\in I.$
\end{de}

\begin{te}\label{thm:ultrast}
A group $G=\mathbb F_\ell\ast\mathbb F_m/\langle W\rangle$ is constraint $\mathcal F$-stable with respect to $a_1^\alpha, \ldots, a_\ell^\alpha\in G_\alpha, \alpha\in I,$
if and only if every group homomorphism $\theta\colon G\rightarrow \prod_{%
k\to\omega}\left( G_{f(k) },d_{f(k) }\right)$ such that $\theta( a_i)=\prod_{k\to\omega}a_i^{f(k)}$ for each $i=1,\ldots, \ell$ is constraint liftable, for any $f\colon\nz\to I$.
\end{te}
\begin{proof}
Let $W\subseteq\mathbb F_\ell\ast\mathbb F_m$ be the system which is constraint $\mathcal F$-stable and let $\theta$ be a group homomorphism as above with $\theta(a_i)=\prod_{k\to\omega}a_i^{f(k)}$ for each $i=1,\ldots, \ell$. Choose maps $\theta_k\colon G\to G_{f(k)}$  such that:
\begin{enumerate}
\item $\theta(g)=\prod_{k\to\omega}\theta_k(g)$, for all $g\in G$;
\item $\theta_k(a_i)=a_i^{f(k)}$ for each $i=1,\ldots, \ell$.
\end{enumerate}
Then $\lim_{k\to\omega}d_{f(k)}\left(w(a_1^{f(k)},\ldots,a_\ell^{f(k)},\theta_k(x_1),\ldots,\theta_k(x_m)),1_{f(k)}\right)=0$ for all $w\in W$. It follows that $\theta_k(x_1),\ldots,\theta_k(x_m)$ are eventually $\delta$-solutions for an arbitrary prescribed $\delta>0$. According to Definition \ref{def:cstab}, there exist homomorphisms $\pi_k\colon\mathbb F_m\to G_{f(k)}$ such that $\pi_k(\bar x_1),\ldots,\pi_k(\bar x_m)$ is a solution to $W$, $\lim_{k\to\omega}d_{f(k)}\left(\pi_k(\bar x_j),\theta_k(x_j)\right)=0$ for each $j=1,\ldots, m,$ and $\pi_k(v_i)=a_i^{f(k)}$ for any $v_i\in\mathbb F_m$ with $\rho(v_i)=a_i$. Then $\tilde\theta\colon G\to\prod_{k\in\nz}(G_{f(k)},d_{f(k)})$, $\tilde\theta(x_j)=\left(\pi_k(\bar x_j)\right)_{k\in\nz}$ is a constraint lift of $\theta$.

For the reverse implication assume that $G$ is not constraint $\mathcal F$-stable. Then there exists $\varepsilon>0$ such that for each $\delta>0$ there is $\alpha\in I$ and $g_1,\ldots,g_m\in G_\alpha$ a $\delta$-solution of $W$ with no $\varepsilon$-close, with respect to $d_\alpha$, solution to $W$, that restricted to $a_i$ is equal to $a_i^\alpha$. Choosing a sequence $\delta_k$ decreasing to $0$, we have such $\delta_k$-solutions $g_1^k,\ldots, g_m^k\in G_{f(k)}$ for some $f\colon \nz\to I$. We assign $x_j\mapsto \prod_{k\to\omega}g_j^k$ for each $j=1,\ldots, m,$ whence a homomorphism $\theta\colon G\to  \prod_{%
k\to\omega}(G_{f(k)},d_{f(k)})$ that admits no constraint lift.
\end{proof}

Extending our result from~\cite[Section 3]{ArPau}, we show that
the definition of (constraint) $\mathcal F$-stability does not depend on the particular choice of finite presentation of the group.

\begin{cor}\label{cor:equiv}
Let $W_1$ and $W_2$ be two equivalent systems. If $W_1$ is constraint $\mathcal F$-stable then so is $W_2$.
\end{cor}
\begin{proof}
It is straightforward, by Theorem \ref{thm:ultrast}.
\end{proof}

Classical fully residually-$\mathcal F$ groups are basic examples of $\mathcal F$-approximable groups.
For instance, residually symmetric groups are examples of sofic groups, residually finite groups are examples of weakly sofic groups, etc. 
Clearly, constraint analogues of fully residually-$\mathcal F$ groups are natural examples of
constraint $\mathcal F$-approximable groups.

\begin{de}[Constraint fully residually-$\mathcal F$ groups]
A group $G$ is \emph{constraint fully residually-$\mathcal F$} with respect to $a_1,\ldots, a_\ell\in G$ and $a_1^\alpha, \ldots, a_\ell^\alpha\in G_\alpha, \alpha\in I$
if for each set of non-identity elements $g_1,\ldots, g_r\in G$ there exists a normal subgroup $K\unlhd G$ such that 
\begin{enumerate}
\item $g_1,\ldots, g_r \not\in K;$
\item $G/K\cong G_\alpha$ for some $\alpha\in I;$
\item $a_iK=a_i^\alpha$ for $i=1,\ldots, \ell.$
\end{enumerate}
\end{de}

Dropping condition (3) above yields the usual definition of fully residually-$\mathcal F$ groups.
Considering one non-identity element (i.e. taking $r=1$), defines the class of (constraint) residually-$\mathcal F$ groups. 
A careful choice of constraints leads easily to examples of groups that are
fully residually finite but that are not constraint fully residually finite. 

\begin{ex}(Integers)
Set $\mathcal F^{cyc}_{\{0,1\}}=((\Z/p\Z)^\times, d_{\{0,1\}}, 1_p)_{p \hbox{ \scriptsize{prime}},\, p\geqslant 3},$
where $(\Z/p\Z)^\times$ is the group of invertible elements $\mod p$, i.e. the cyclic group of order $p-1$,  with identity element $1_p$, and 
with the trivial $\{0, 1\}$-valued metric $d_{\{0,1\}}$ (induced by the length function assigning  length 1
to each non-trivial group element).

Obviously, $\Z$ is fully residually-$\mathcal F^{cyc}_{\{0,1\}}$.
For each $p$, choose $a^p\in \Z/p\Z$, an element that is a quadratic nonresidue $\mod p$. Then, using the above notation (for $\ell=1$ and $\alpha=p$), $\Z$ is not constraint fully residually-$\mathcal F^{cyc}_{\{0,1\}}$ with respect to
$a=2\in \Z$ and $a^p\in \Z/p\Z$.

\end{ex}

Next we express robustness of (constraint)  $\mathcal F$-approximable  $\mathcal F$-stable groups.
It extends our previous result on stable sofic groups~\cite[Theorem 4.3]{ArPau} to arbitrary (constraint) metric approximations. 

We write that $\mathcal S\mathcal F=\mathcal F$ if all subgroups of every $G_{\alpha}\in\mathcal F$ 
belong to $\mathcal F$.

\begin{te} Let $\mathcal S\mathcal F=\mathcal F$.
If $G=\mathbb F_\ell\ast\mathbb F_m/\langle W\rangle$ is both constraint $\mathcal F$-approximable and constraint $\mathcal F$-stable,
then $G$ is constraint fully residually-$\mathcal F.$  
\end{te}
\begin{proof}
Since $G$ is constraint $\mathcal F$-approximable, by Theorem \ref{thm:ultra}, there exists an injective group homomorphism $$\theta\colon G\hookrightarrow\prod_{k\to\omega}(G_{f(k)},d_{f(k)})$$ such that $\theta(a_i)=\prod_{k\to\omega}a_i^{f(k)}$ for  $i=1,\ldots, \ell$. Since $G$ is constraint $\mathcal F$-stable, by Theorem \ref{thm:ultrast}, there exists $\tilde\theta\colon G\to\prod_{k\in\nz}G_{f(k)}$, a constraint lift of $\theta$. 
Given finitely many non-identity elements $g_1,\ldots, g_r\in G$, it remains to choose $k\in\nz$ such that $\tilde\theta(g_j)$ is non-identity for $j=1,\ldots,r$. This is possible as $\theta$ and, thus, $\tilde\theta$ are injective homomorphisms.
Since $\mathcal S\mathcal F=\mathcal F$, their images are groups from $\mathcal F$. 
\end{proof}

\section{Examples of non constraint sofic approximations}\label{sec:sofic}

Let us recapitulate the context of constraint metric approximability: 
$\mathcal F=(G_\alpha,d_\alpha,\ve_\alpha)_{\alpha\in I}$ is a family of groups, $\ell\geqslant 0,$ and we fix $a_1^\alpha,\ldots,a_\ell^\alpha\in G_\alpha$ for each $\alpha\in I$. Further, $G$ is an arbitrary countable group where we  fix $\ell$ group elements $a_1,\ldots, a_\ell\in G$. We denote this data by $(G\restriction a_1, \ldots, a_\ell)$ and this is the object we want to constraint $\mathcal F$-approximate or, in contrast, for which we show that it admits no any constraint $\mathcal F$-approximation by $(G_\alpha \restriction a_1^\alpha,\ldots,a_\ell^\alpha)_{\alpha\in I}$.

Let $H\leqslant G$ be the subgroup generated by $a_1,\ldots,a_\ell\in G$. If the group $H$ cannot be $\mathcal F$-approximated (without constraints) using only 
the coefficients $a_i^\alpha,\, i=1, \ldots, \ell$, then $G$ is not constraint $\mathcal F$-approximable by $(G_\alpha \restriction a_1^\alpha,\ldots,a_\ell^\alpha)_{\alpha\in I}$ in a trivial way. Thus, when constructing a meaningful non-trivial counter-example to the existence of a constraint $\mathcal F$-approximation of $G$, we have to make sure that there indeed exist homomorphisms $\theta\colon H\to\prod_{k\to\omega}(G_{f(k)},d_{f(k)})$ such that $\theta(a_i)=\prod_{k\to\omega}a_i^{f(k)}$.  Thus, the subgroup $H\leqslant G$ is viewed as the `fixed part', the one for which the $\mathcal F$-approximation is already given and cannot be changed. Our concept of constraint $\mathcal F$-approximability makes  rigorous 
the analysis of whether or not
such a `fixed' approximation can be extended from a subgroup $H$ to the ambient group $G$.

We produce now an example where there is no such an extension. This is done in the realm of
 \emph{constraint sofic} approximations. That is,  the approximating family is   $\mathcal F^{sof}=(S_n, d_H, 1_n)_{n\in \mathbb N},$
where $S_n$ denotes the \emph{symmetric group} acting on the set $\{1,\ldots,n\}$, with the identity element $1_n\in S_n$, and $d_H$ denotes
the \emph{normalised Hamming distance} defined, for two elements $p,q\in S_n$, by 
$$d_H(p,q)=\frac1nCard\left\{i:p(i)\neq q(i)\right\}.$$ 
The metric ultraproduct of $S_{n_k}, k\in\nz$ with respect to the normalized Hamming distance is the \emph{universal sofic group}, an object introduced by Elek-Szabo~\cite{ElSza}: $$\Pi_{k\to\omega}S_{n_k}=\Pi_{k\in\nz}S_{n_k}\slash\{(p_k)_{k\in\nz}\in\Pi_{k\in\nz}S_{n_k}:\lim_{k\to\omega}d_H(p_k,1_{n_k})=0\},$$
endowed as usual with the canonical bi-invariant metric $d_\omega$.

The following result  provides an explicit example of a group which is not constraint approximable by a subfamily of $\mathcal F^{sof}$, in a non-trivial way (cf. Example~\ref{ex:yes}). 

\begin{te}\label{thm:ncs}
There exist $a^{n_k}_1, a_2^{n_k}\in S_{n_k}, k\in \mathbb N$ such that
the group $\mathbb F_2\times\Z=\langle a_1, a_2, x \mid [a_1, x ]=1, [a_2, x]=1\rangle$
 is not constraint $(S_{n_k}, d_H, 1)_{k\in \mathbb N}$-approximable with respect to $a^{n_k}_1, a_2^{n_k}, k\in \mathbb N$.
\end{te}

Here, $\mathbb F_2=\langle a_1, a_2\rangle$ is the `fixed subgroup' of $\mathbb F_2\times\Z$. Our strategy to prove Theorem~\ref{thm:ncs} is to show 
the existence of a suitable sofic representation of $\mathbb F_2$. Recall that a
\emph{sofic representation} of a group $G$ is 
a homomorphism $\theta\colon G\to \Pi_{k\to\omega}S_{n_k}$ with  $d_\omega(\theta(g),1_\omega)=1$ for every element $1_G\not=g\in G$.

\begin{te}\label{thm:trivialcommutant}
There exists a sofic representation of the free group $\mathbb F_2,$
$$\theta\colon\mathbb F_2\hookrightarrow\Pi_{k\to\omega}S_{n_k}$$ such that its commutant $\theta(\mathbb F_2)^\prime=\{p\in\Pi_{k\to\omega}S_{n_k}:p\theta(w)=\theta(w)p,\ \forall w\in\mathbb F_2\}$ is trivial, i.e. contains only the identity element $1_{\omega}$.
\end{te}
 
In order to construct such a sofic representation we use
two results from \cite{Pa}. We temporarily fix an integer $n>0$, the degree of the symmetric group $S_n$. Denote by $a\in S_n$ the $n$-cycle with $a(i)=i+1$ and $a(n)=1$.

\begin{p}\cite[Proposition 5.13]{Pa}\label{P5.13}
Let $\ve>0$. The number of permutations $y\in S_n$ such that $d_{H}(ay,ya)<\ve$ is less than $n^{\lfloor n\ve\rfloor+1}$.
\end{p}

\begin{p}\cite[Theorem 5.20]{Pa}\label{T5.20}
For any $\ve>0$ and $w\in\mathbb F_2$, there exists $n_0$ such that for any $n>n_0$ for at least $(1-\ve)[(n-1)!]$ $n$-cycles $c\in S_n$ we have 
$d_H\left(w(a,c), 1_n\right)>1-\ve$.
\end{p}

Proposition \ref{P5.13} provides an estimate for the number of permutations almost commuting with an $n$-cycle. However, for our construction, we also need an estimate for the number of permutations commuting with an arbitrary element.

\begin{p}\label{P:nr commuting}
Let $b\in S_n$ be such that $d_H(b, 1_n)>4\ve$. The number of permutations $c\in S_n$ such that $d_H(bc,cb)<\ve$ is less than $\frac{n!}{n^{n\ve+3}}$, for large enough $n$.
\end{p}
\begin{proof}
Let $\delta=4\ve$ and define $C=\{c\in S_n: d_H(bc,cb)<\ve\}$. Choose $c\in C$. Consider the following subsets of $\{1,\ldots,n\}$: $A_c=\{i:bc(i)=cb(i)\}$ and $B=\{i:b(i)\neq i\}$. Then $|A_c|>(1-\ve)n$ and $|B|>\delta n$. It follows that $|A_c\cap B|>(\delta-\ve)\cdot n$.

Let $i\in A_c\cap B$. Then $c\big(b(i)\big)=bc(i)$ and $b(i)\neq i$. Hence, once the value of $c(i)$ is fixed, the value of $c$ on $b(i)$ must be $bc(i)$. Unfortunately, the set $A_c$ depends on $c$. This makes the counting argument a little more involved.

Lets recall how to count the number of permutations $p\in S_n$: $p(1)$ can take any of the $n$ values in the set $\{1,\ldots,n\}$; $p(2)$ can take any of the remaining $n-1$ values, and so on. Hence, the cardinality of $S_n$ is $n!$. We adapt this argument to count the number of permutations $c$ with the required properties. Without loss of generality, we can assume that $B=\{1,2,\ldots,|B|\}$. As before $c(1)$ can take $n$ values. If $1\in A_c$, a information that at the moment we don't have, than $c(b(1))$ is also set. Thus, the following value of $c$ to be decided ($c(2)$ if $b(2)\neq 2$, and $c(3)$ otherwise), has only $n-2$ options. If $1\notin A_c$, we continue our enumeration of elements in $B$ till $|B|$. In the worst scenario, the first $\ve\cdot n$ elements of $B$ will not be in $A_c$. After this, all remaining elements of $B$ are bound to also be in $A_c$.

Thus, denoting by $t=\lfloor\ve n\rfloor$ and $s=\lfloor (\delta-\ve)n/2\rfloor$, our estimation for the maximal number of elements in $C$ is: $$\underbrace{n(n-1)\ldots(n-t+1)}_\text{$t$ terms}\underbrace{(n-t)(n-t-2)\ldots(n-t-2s+2)}_\text{$s$ terms}(n-t-2s)(n-t-2s-1)\ldots1.$$ Hence:
\[|C|<\frac{n!}{(n-t-2s+1)^s}<\frac{n!}{[(1-\delta)n]^{(\delta-\ve)n/2-1}}.\]
We only need to show that $[(1-\delta)n]^{(\delta-\ve)n/2-1}>n^{n\ve+3}$. Using the logarithm, this is equivalent to:
\[((\delta-\ve)n/2-1)\ln[(1-\delta)n]>(n\ve+3)\ln{n}.\]
We factor the two terms and compute the limit via L'Hospital's rule.
\begin{align*}
&\lim_{n\to\infty}\frac{((\delta-\ve)n/2-1)\ln[(1-\delta)n]}{(n\ve+3)\ln{n}}=\lim_n\frac{((\delta-\ve)/2)\ln[(1-\delta)n]+((\delta-\ve)n/2-1)\cdot1/n}{\ve \ln{n}+(n\ve+3)\cdot 1/n}=\\
&\lim_n\frac{((\delta-\ve)/2)\ln[(1-\delta)n]}{\ve \ln{n}}=\frac{((\delta-\ve)/2)}{\ve}=\frac{4\ve-\ve}{2\ve}=\frac32>1.
\end{align*}
\end{proof}

We continue our counting argument by introducing two sets of $n$-cycles with specific properties. Given  $\delta >0$, we define:
\[L_n^\delta=\{c\in S_n,c\mbox{ is an $n$-cycle} :\not\exists b\in S_n\mbox{ with } d_H(b,1_n)>4\delta, d_H(ab,ba)<\delta\mbox{ and }d_H(cb,bc)<\delta\}.\]

\begin{p}\label{P:Gndelta}
For a fixed $\delta>0$ and large enough $n\in\nz$, $$Card\ L_n^\delta>(1-n^{-1})[(n-1)!].$$ 
\end{p}
\begin{proof}
According to Proposition \ref{P5.13} there are at most $n^{n\delta+1}$ permutations $b\in S_n$ such that $d_H(ab,ba)<\delta$. By Proposition \ref{P:nr commuting}, for each of those permutations $b$ with $d_H(b,1_n)>4\delta$, there are at most $n!\cdot n^{-n\delta-3}$ cycles $c$ such that $d_H(cb,bc)<\delta$. All in all, the complement of $L_n^\delta$ has a cardinality less than $n^{n\delta+1}\cdot n!\cdot n^{-n\delta-3}=n^{-1}(n-1)!$.  The conclusion hence follows.
\end{proof}

The set $L_n^\delta$ cannot be used directly to construct the required sofic representation. This is because $d_H(b,1_n)$ is in some sense a moving target, while in the definition of $L_n^\delta$ it is supposed to be fixed. This is why we introduce the following set:
\[K_n^\delta=\{c\in S_n,c\mbox{ is an $n$-cycle }:\forall b\in S_n,\ d_H(b,1_n)\leqslant 8\cdot max\{d_H(ab,ba),d_H(bc,cb),\delta\}.\}\]

\begin{p}\label{P:Hndelta}
For a fixed $1>\delta>0$ and a large enough $n\in\nz$, $$Card\ K_n^\delta>(1-\delta)[(n-1)!].$$ 
\end{p}
\begin{proof}
The proof is almost over when we notice that $K_n^\delta\supseteq L_n^\delta\cap L_n^{2\delta}\cap\ldots\cap L_n^{2^k\delta}$, where $k$ is minimal with the property that $2^{k+2}\delta>1$. So let $c\in L_n^\delta\cap L_n^{2\delta}\cap\ldots\cap L_n^{2^k\delta}$ and take $b\in S_n$. Denote by 
 $\lambda=max\{d_H(ab,ba),d_H(bc,cb)\}$. If $\lambda<\delta$, then as $c\in L_n^\delta$, $d_H(b,1_n)\leqslant 4\delta$.

Assume that $\lambda\geqslant\delta$. Then, there exists $i>0$ such that $2^{i-1}\delta\leqslant\lambda<2^i\delta$. If $i\leqslant k$, then $c\in L_n^{2^i\delta}$, so $d_H(b,1_n)\leqslant 4\cdot2^i\delta\leqslant 8\lambda$. If $i>k$ then $8\lambda>1$. This proves $c\in K_n^\delta$.

By Proposition \ref{P:Gndelta} and using De Morgan's formula $\cap_{j=0}^{k}L_n^{2^j\delta}=\overline{\cup_{j=0}^{k}\overline{L_n^{2^j\delta}}},$ we obtain that
$$|L_n^\delta\cap L_n^{2\delta}\cap\ldots\cap L_n^{2^k\delta}|>(1-(k+1)n^{-1})[(n-1)!].$$ As $k$ is fixed, depending only on $\delta$, we can find $n$ such that $|K_n^\delta|>(1-\delta)[(n-1)!]$.
\end{proof}

Now we construct the required sofic representation of $\mathbb F_2$, so $n$ is no longer fixed. 
We denote  by $a^n_1\in S_n$ the canonical $n$-cycle $i\mapsto i+1$, $n\mapsto 1$. (The notation $a^n_1\in S_n$ reflects our use of $a_1^\alpha, \ldots, a_\ell^\alpha\in G_\alpha$ in a general setting.)

\begin{p}\label{f2rep}
For each $\delta>0$ there exist $n\in\nz^*$ and $\pi\colon\mathbb F_2=\langle a_1, a_2\rangle \to S_n$ such that:
\begin{enumerate}
\item[(1)] $d_H\left( w(\pi(a_1),\pi(a_2)), 1_n \right) >1-\delta$ for every $1_{\mathbb F_2}\not=w\in \mathbb F_2$ of length at most $1/\delta$, 

\item[(2)] $\pi(a_1)=a^n_1$ and $\pi(a_2)\in K_n^\delta$.
\end{enumerate}
\end{p}
\begin{proof}
Given $\delta>0$, denote by $B_{\mathbb F_2}^k$ the set of words in $\mathbb F_2$ of length at most $k=\lfloor1/\delta\rfloor$ and choose $\ve>0$, $\ve\leqslant\delta$ such that $\ve\cdot Card\ B_{\mathbb F_2}^k+\delta<1$. By Propositions \ref{T5.20}, applied to each $1_{\mathbb F_2}\not=w\in B_{\mathbb F_2}^k$, for large enough $n$, for at least $(1-\ve\cdot Card\ B_{\mathbb F_2}^k)[(n-1)!]$ $n$-cycles $c\in S_n$ we have $d_H(w(a^n_1,c),1_n)>1-\ve$. For this estimate, we again use De Morgan's formula, as in the proof of Propostion \ref{P:Hndelta}.

Now, using the statement of Proposition \ref{P:Hndelta}, for large enough $n\in\nz$, for at least $(1-\ve\cdot Card\ B_{\mathbb F_2}^k-\delta)[(n-1)!]$ $n$-cycles $c\in S_n$, we have the previous property and also $c\in K_n^\delta$. Choosing such a cycle $c\in S_n$ and setting $\pi(a_2)=c$ yields conditions $(1)$ and $(2)$ above.
\end{proof}

We are ready now to prove Theorems~\ref{thm:trivialcommutant} and ~\ref{thm:ncs}.

\begin{proof}[Proof of Theorem~\ref{thm:trivialcommutant}]
Let $(\delta_k)_{k\in\mathbb N}\in\mathbb R_+^*$ be a decreasing sequence converging to $0$ as $k\to\infty$. Using the previous Proposition, for each $k\in\mathbb N$, construct $\pi_k\colon\mathbb F_2\to S_{n_k}$ with the stated properties for $\delta=\delta_k$. Then construct $\theta=\Pi_{k\to\omega}\pi_k\colon\mathbb F_2\to\Pi_{k\to\omega}S_{n_k}$. By the first condition, $\theta$ is a sofic representation.

Let $b=\Pi_{k\to\omega}b_k\in\Pi_{k\to\omega}S_{n_k}$ be in the commutant of $\theta$. Then we have that $\lim_{k\to\omega}d_H(\pi_k(a_1)b_k,b_k\pi_k(a_1))=0$ and $\lim_{k\to\omega}d_H(\pi_k(a_2)b_k,b_k\pi_k(a_2))=0$. Let $\ve>0$. There exists $F\in\omega$ such that for all $k\in F$, $\delta_k<\ve$ and $d_H(\pi_k(a_1)b_k,b_k\pi_k(a_1))<\ve$, $d_H(\pi_k(a_2)b_k,b_k\pi_k(a_2))<\ve$. As $\pi_k(a_2)\in K_{n_k}^{\delta_k}$, we get $$d_H(b_k,1_{n_k})\leqslant 8\cdot max\{d_H(\pi_k(a_1)b_k,b_k\pi_k(a_1)), d_H(\pi_k(a_2)b_k,b_k\pi_k(a_2)), \delta_k\}<8\ve.$$ It follows that $d_H(b,1_\omega)\leqslant 8\ve$. As $\ve$ is arbitrary, $b=1_\omega$.
\end{proof}

\begin{proof}[Proof of Theorem~\ref{thm:ncs}] 
Using Proposition \ref{f2rep}, we construct $\pi_k\colon\mathbb F_2=\langle a_1, a_2\rangle \to S_{n_k}$ such that:
\begin{enumerate}
\item[(1)] $d_H\left( w(\pi_k(a_1),\pi_k(a_2)), 1_{n_k} \right) \to1$ as $k\to \infty$ for every $1_{\mathbb F_2}\not=w\in \mathbb F_2$ of length at most $k$ (this is the classical limit, not the $\omega$-limit);

\item[(2)] $\pi_k(a_1)=a^{n_k}_1$ and $a^{n_k}_2=\pi_k(a_2)\in K_{n_k}^{1/k}$.
\end{enumerate}

In order to show that $\mathbb F_2\times\Z=\langle a_1, a_2, x \mid [a_1, x ]=1, [a_2, x]=1\rangle$
 is not constraint $(S_{n_k}, d_H, 1_{n_k})_{n\in \mathbb N}$-approximable with respect to $a^{n_k}_1, a_2^{n_k}, n\in \mathbb N$, we use Theorem \ref{thm:ultra}.
 
Suppose to the contrary, and let $\omega$, $f\colon\mathbb N\to\mathbb N,$ and $\theta\colon\mathbb F_2\times\Z\to\Pi_{k\to\omega}S_{n_{f(k)}}$ be given by that theorem.  Then $\lim_{k\to\omega}f(k)=\infty$, otherwise the ultraproduct would be a finite group. It follows that $d_H(\theta(w),1_\omega)=1$ for each non-trivial $w$. Using similar arguments as in the proof of Theorem \ref{thm:trivialcommutant}, we also get that $\theta(\mathbb F_2)$ has trivial commutant.

Let $b=\theta(x)$, where $x$ is the generator of the subgroup $\mathbb Z$. As $b$ is in the commutant of $\theta(\mathbb F_2)$, $b$ needs to be the identity. This is in contradiction with $d_H(\theta(x),1_\omega)=1$.
\end{proof}

This non-trivial example is due to a careful choice of coefficients $a_1^{n_k}, a_2^{n_k}\in S_{n_k}$. If we allow extra space in approximating direct products, 
such as $G\simeq\mathbb F_2\times \Z$, the group becomes constraint $\mathcal F^{sof}$-approximable as our next example shows. 

\begin{ex}[Direct products]\label{ex:yes}
Let $H=\langle a_1,\ldots,a_\ell\rangle$. Fix $\theta\colon H\hookrightarrow\prod_{n\to\omega}S_{n_k}$ a sofic representation and $\theta_k\colon H\to S_{n_k}$ such that $\theta=\prod_{k\to\omega}\theta_k$. Then $\big(H\times G\restriction a_1,\ldots, a_\ell\big)$ is constraint approximable by 
$\big(S_{(n_k)^2}\restriction \theta_k(a_1)\otimes 1_{n_k},\ldots, \theta_k(a_\ell)\otimes 1_{n_k}\big),$ 
whenever $G$ is an arbitrary countable sofic group.
\end{ex}

However, there are examples of pairs of groups, $H\leqslant G$, where any amplification will not help in making the group $G$ constraint approximable with respect to $H$. Such an example is provided in Section 5 of \cite{Pa}. A sofic representation of $\mathbb F_2=\langle a_1, a_2\rangle$ is constructed, that even when amplified it cannot be extended to a sofic representation of $\Z\ast\Z/2\Z=\langle a_1, a_2, x \mid x^2=1,xa_1x=a_2\rangle$.

\begin{te}\cite[Proposition 5.21]{Pa}
There exist $a^{n_k}_1, a_2^{n_k}\in S_{n_k}, k\in \mathbb N$ such that
the group $\Z\ast\Z/2\Z=\langle a_1, a_2, x \mid x^2=1,xa_1x=a_2\rangle$
 is not constraint $(S_{n_k^2}, d_H, 1)_{k\in \mathbb N}$-approximable by $\big(S_{(n_k)^2}\restriction a^{n_k}_1\otimes 1_{n_k}, a^{n_k}_2\otimes 1_{n_k}\big)$.
\end{te}

Our method of proving Theorems~\ref{thm:trivialcommutant} and ~\ref{thm:ncs} is very relevant to Question~2.14 from \cite{Pa1}
to which we turn our attention now.  The remainder of this section assumes familiarity with \cite{Pa1} and, to a lesser extent, \cite{Pa}. The space of sofic representations $Sof(G,P^\omega)$ is introduced in  \cite[Section 1.4]{Pa1} and its convex structure is defined in \cite[Section 2.2]{Pa1}. All the notation used is introduced in Sections 1.1 and 1.2 of the same article. For the proof we also need some techniques from \cite[Section 5]{Pa}. In particular, the normalized Hamming distance on arbitrary matrices, and the notion of expander in this context are important. 

\begin{de}\cite[Definition 5.1]{Pa}
For $x,y\in M_n(\mathbb C)$ we define:
\[d_H(x,y)=\frac1nCard\left\{i:\exists j\ x(i,j)\neq y(i,j)\right\}.\]
\end{de}

This definition is consistent with the normalised Hamming distance on permutations we  used so far.

\begin{de}
A sequence of pairs of permutations $\{a_1^k,a_2^k\}_{k\in \mathbb N}$, $a_1^k,a_2^k\in P_{n_k}$, is an \emph{expander} if there exists $\lambda>0$ such that for any $k$ and any projection $p\in D_{n_k}$ with $0<Tr(p)\leqslant\frac12$ we have $\lambda Tr(p)<d_H\big(p,a_1^kp(a_1^k)^*\big)+d_H\big(p,a_2^kp(a_2^k)^*\big)$ (see condition (1) of  \cite[Proposition 5.5]{Pa}).
\end{de}

\begin{p}
Let $\theta\colon\mathbb F_2=\langle a_1, a_2\rangle\hookrightarrow\Pi_{k\to\omega}S_{n_k}$ be a sofic representation. Choose $a_1^k,a_2^k\in S_{n_k}$ such that $\theta(a_1)=\Pi_{k\to\omega}a_1^k$ and $\theta(a_2)=\Pi_{k\to\omega}a_2^k$. Assume that $\{a_1^k,a_2^k\}_{k\in \mathbb N}$ is an expander. Then $[\theta]$ is an extreme point in $Sof(\mathbb F_2,P^\omega)$.
\end{p}
\begin{proof}
We use Lemma 2.12 from \cite{Pa1}. Let $\{r_k\}_{k\in \mathbb N}$ be a sequence of natural numbers and $p=\Pi_{k\to\omega}p_k\in (\theta\otimes 1_{r_k})'\cap\Pi_{k\to\omega}D_{n_kr_k}$ be a projection. We show that $p=\Pi_{k\to\omega}q_k$, where $q_k=Id_{n_k}\otimes s_k$, for some projection $s_k\in D_{r_k}$. 

As $p_k\in D_{n_kr_k}$, there exist projections $p_k^i\in D_{n_k}$ such that $p_k=p_k^1\oplus\ldots\oplus p_k^{r_k}$. We construct $q_k\in D_{n_kr_k}$, by replacing each projection $p_k^i$ with $0_{n_k}$ or $Id_{n_k}$, depending on which one is closer. So $q_k=q_k^1\oplus\ldots\oplus q_k^{r_k}$, with $q_k^i\in\{0_{n_k},Id_{n_k}\}$ and $d_H(p_k^i,q_k^i)=min\{Tr(p_k^i),1-Tr(p_k^i)\}$. Thus $d_H(p_k,q_k)=\frac1{r_k}\sum_{i=1}^{r_k}min\{Tr(p_k^i),1-Tr(p_k^i)\}$. 

From the definition of an expander, we can deduce that $\lambda min\{Tr(s),1-Tr(s)\}<d_H\big(s,a_1^ks(a_1^k)^*\big)+d_H\big(s,a_2^ks(a_2^k)^*\big)$ for any projection $s\in D_{n_k}$. It follows that:
\begin{align*}
\lambda d_H(p_k,q_k)<&\frac1{r_k}\sum_{i=1}^{r_k}d_H\big(p_k^i,a_1^kp_k^i(a_1^k)^*\big)+d_H\big(p_k^i,a_2^kp_k^i(a_2^k)^*\big)\\
=&d_H\big(p_k,(a_1^k\otimes 1_{r_k})p_k(a_1^k\otimes 1_{r_k})^*\big)+d_H\big(p_k,(a_2^k\otimes 1_{r_k})p_k(a_2^k\otimes 1_{r_k})^*\big).
\end{align*}
Setting $q=\Pi_{k\to\omega}q_k$, and passing to the ultralimit, we get:
\[\lambda d_H(p,q)\leqslant d_H(p,(\theta(a_1)\otimes 1_{r_k})p(\theta(a_1)\otimes 1_{r_k})^*)+d_H(p,(\theta(a_2)\otimes 1_{r_k})p(\theta(a_2)\otimes 1_{r_k})^*).\]
As $p\in (\theta\otimes 1_{r_k})'$, the last term is equal to zero. It follows that $d_H(p,q)=0$, so $p=q$. 

By construction, $q=1_{n_k}\otimes s$, with $s$ a projection in $\Pi_{k\to\omega}D_{r_k}$. Then $(\theta\otimes 1_{r_k})_q$ is just an amplification of $\theta$ (check Definition 2.4 from \cite{Pa1}). Then $[(\theta\otimes 1_{r_k})_q]=[\theta]$ and we are done.
\end{proof}

\begin{rem}
Elaborating on these arguments, one can prove that $[\theta]$ is an extreme point even when the expander condition is replaced by the weaker assumption that the induced action of $\theta$ on the Loeb space is ergodic.
\end{rem}

\begin{cor}
There exists a sofic representation of the free group $\mathbb F_2,$
$$\theta\colon\mathbb F_2\hookrightarrow\Pi_{k\to\omega}S_{n_k},$$ such that $[\theta]$ is an extreme point of the convex structure $Sof(\mathbb F_2,P^\omega)$, and such that its commutant is trivial. This solves in negative Question 2.14 from \cite{Pa1}.
\end{cor}
\begin{proof}
We follow the same plan as in the proof of Theorem \ref{thm:trivialcommutant}. Only Proposition~\ref{f2rep} has to be slightly adapted. Namely, when choosing $c\in S_n$, we make sure that  $a_1^n$ and $c$ satisfy the expander formula for $\lambda=0.2$. This can be done by Proposition 5.11 from \cite{Pa}. The ``$a,c$" notation used there is consistent with the present article. The sofic representation so obtained, has trivial commutant and it's an extreme point by the previous proposition.
\end{proof}

\section{Centralizer equation in symmetric groups and soficity}\label{sec:centra}

In this section, we investigate the constraint $\mathcal F^{sof}$-stability for the commutator equation:
$\ell=1, m=1, W=[a,x],$ where $a$ is a coefficient and $x$ is a variable.

Let $X=\{1,\ldots,n\}$.  For $p\in S_n$, a \emph{cycle} of $p$ is a subset $c=\{x_1,\ldots, x_t\}\subseteq X$ on which the action of $p$ is \emph{ergodic}, that is, $\{x_1,\ldots,x_t\}=\{p(x_1),p^2(x_1),\ldots,p^t(x_1)\}$. Then $Card(c)$ plays the role of the length of the cycle $c$. For $i\in\nz^*$ and $p\in S_n,$ we define 
$$cyc_i(p)=\frac1nCard\left\{x\in X:x\in c,Card(c)=i\right\},$$ that is, the ratio of elements in $X$ that are part of cycles of length $i$ in $p$. We have $\sum_icyc_i(p)=1$.
For $i\in\nz^*$ and $p=\Pi_{k\to\omega}p_k,$ we define $$cyc_i(p)=\lim_{k\to\omega}cyc_i(p_k).$$ As $\sum_icyc_i(p_k)=1$ for every $k$, it follows that $\sum_icyc_i(p)\leqslant 1$. Also set: $$cyc_\infty(p)=1-\sum_icyc_i(p)\geqslant 0.$$ As in the case of symmetric group $S_n$, the numbers $cyc_i(p)$, associated to $p=\Pi_{k\to\omega}p_k$, constitute a complete set of invariants under conjugacy equivalence relation:

\begin{p}\cite[Proposition 2.3(4)]{ElSza}\label{conjugacyclasses}
Two elements $p,q\in\Pi_{k\to\omega}S_{n_k}$ are conjugate if and only if $cyc_i(p)=cyc_i(q)$ for all $i\in\nz^*$.
\end{p}

\subsection{An almost centralizing permutation away from the centralizer}
Let $p\in S_n$ be an arbitrary permutation and $c=\{x_1,\ldots,x_t\}$ be a cycle of $p$. Let us assume that $p^i(x_1)=x_{i+1}$, for $i=1,\ldots, t-1$. We want to construct a permutation, supported on $c$, almost commuting with $p$, and that is far from the centralizer of $p$. Construct $r_{p, c}\in S_n$ as follows: let $a=\lfloor t/3\rfloor$ and define $r_{p,c}(x_j)=x_{j+a}$ for $j=1,\ldots,a$; $r_{p,c}(x_j)=x_{j-a}$ for $j=a+1,\ldots,2a$; $r_{p, c}(x_j)=x_j$ for $j=2a+1,\ldots, t$ and $r_{p,c}(x)=x$ for $x\notin c$. So, $r_{p,c}(c)=c$ and $r_{p,c}$ is the identity outside of $c$. Then $d_H(p\cdot r_{p,c},r_{p,c}\cdot p)=3/n$. Since every permutation commuting with a cycle is a power of that cycle, we deduce that for any $q\in S_n$ commuting with $p$, we have $d_H(q,r_{p,c})\geqslant 2a/n$.

This example is rather general (our $p$ is an arbitrary permutation) and shows also that
the study of $\mathcal F$-stability as we performed in~\cite{ArPau} is very different from investigations of constraint $\mathcal F$-stability as
we initiate in the present paper. Indeed, one of the main results of~\cite{ArPau} is $\mathcal F^{sof}$-stability of the commutator (equation). In contrast,
the preceding example shows that fixing one of the commuting elements (i.e. imposing the constraint) change the $\mathcal F^{sof}$-stability property drastically.
In the next subsection, we characterize permutations which are constraint $\mathcal F^{sof}$-stable, when considering the commutator equation.

\subsection{Stability of centralizer equation}
\begin{de}[Stable centralizer]
Let $(p_k)_{k\in\nz}\in\Pi_{k\in\nz} S_{n_k}$ be a sequence of permutations and $p=\Pi_{k\to\omega}p_k\in\Pi_{k\to\omega}S_{n_k}$ its image in the universal sofic group. We say that $(p_k)_{k\in\nz}$ has \emph{stable centralizer}  in permutations if for any $q\in\Pi_{k\to\omega}S_{n_k}$ such that $pq=qp$ there exists $q_k\in S_{n_k}$ such that $q=\Pi_{k\to\omega}q_k$ and $p_kq_k=q_kp_k$ for every $k\in\nz$. 
\end{de}

Thus, stable centralizer is an instance of constraint $\mathcal F^{sof}$-stability for the commutator equation:
$\ell=1, m=1, W=[a,x],$ and $G=\mathbb Z \times \mathbb Z=\langle a,x \mid [a,x]=1\rangle $ in the notation of Definitions~\ref{def:cstab} and~\ref{def:gcstab}.

\begin{te}\label{thm:centra}
A sequence $(p_k)_{k\in\nz}\in\Pi_{k\in\nz} S_{n_k}$ has stable centralizer if and only if $$cyc_\infty(\Pi_{k\to\omega}p_k)=0.$$
\end{te}
\begin{proof}
\textbf{Direct implication.} 
Let us first consider $(p_k)_{k\in\nz}\in\Pi_{k\in\nz} S_{n_k}$ such that $cyc_\infty(p)>0$, where $p=\Pi_{k\to\omega}p_{k}$. Strictly positive $cyc_\infty(p)$ can happen if there are fewer and fewer cycles in $p_k$, as $k\to\omega$, that occupy the same space. In order to see this, we want to split the cycles of $p_k$ in two categories: those that add to $\sum_{i\in\nz^*}cyc_i(p)$ and those that should be considered as part of $cyc_\infty(p)$. The permutation $p_k$ has exactly $\frac{cyc_i(p_k)\cdot n_k}{i}$ cycles of length $i$. If this number is much larger than $\frac{cyc_i(p)\cdot n_k}{i}$ then some of these cycles have to be discarded towards $cyc_\infty(p)$. The formal definition is as follows. For $i,k\in\nz$ define:
\[c(i,k)=\min\left\{ \frac{cyc_i(p_k)\cdot n_k}{i},\lfloor\frac{cyc_i(p)\cdot n_k}{i}\rfloor\right\}.\]
Let $P(k)$ be a collection of cycles of $p_k$ that contains exactly $c(i,k)$ cycles of length $i$ and let
$E(k)$ be the collection of the remaining cycles. The goal here is to prove that $\lim_{k\to\omega}e(k)/n_k=0$, where $e(k)=Card(E(k))$ (this is similar to the proof of Elek and Szabo that numbers $cyc_i(p)$ completely describe the conjugacy class of $p$ inside $\Pi_{k\to\omega}S_{n_k}$). Let $e(i,k)$ be the number of cycles of length $i$ in $E(k)$, such that $e(k)=\sum_ie(i,k)$. We know from the definition of $cyc_i(p)$ that $\lim_{k\to\omega}e(i,k)/n_k=0$. Then for any $T\in\nz^*$:
\[\lim_{k\to\omega}\frac{e(k)}{n_k}=\lim_{k\to\omega}\sum_{i<T}\frac{e(i,k)}{n_k}+ \lim_{k\to\omega}\sum_{i\geqslant T}\frac{e(i,k)}{n_k}\leqslant\frac1T.\]
This inequality holds for any $T\in\nz^*$ so indeed $\lim_{k\to\omega}e(k)/n_k=0$. In conclusion, cycles in $E(k)$ are trivial in number, but not in size, as their total relative length is $cyc_\infty(p)$. It is this phenomena that makes $p$ unstable with respect to its centraliser.

We now construct $q=\Pi_{k\to\omega}q_k$ to contradict the stable centraliser property of $p$. Define $q_k=\Pi_{c\in E(k)}r_{p_k,c}$. Then:
\[d_H(p_kq_k,q_kp_k)=\sum_{c\in E(k)}d_H(p_kr_{p_k,c},r_{p_k,c}p_k)=\sum_{c\in E(k)}\frac3{n_k}=\frac{3e(k)}{n_k}.\]
As $\lim_{k\to\omega}e(k)/n_k=0$, it follows that $pq=qp$. Let now, for every $k\in\nz$, $s_k\in S_{n_k}$ be a permutation commuting with $p_k$. We want to evaluate $d_H(q_k,s_k)$. For $c\in E(k)$, we have $q_k(c)=c$. The permutation $s_k$ is sending cycles of $p_k$ into cycles of $p_k$, as it is commuting with $p_k$. Therefore, in order to minimise $d_H(q_k,s_k)$, we can assume that $s_k(c)=c$. By the definition of $q_k$, for any $c\in E(k)$, the permutations $s_k$ and $q_k$ have to differ on at least $2\cdot \lfloor {Card(c)}/3\rfloor$ points. Then:
\[d_H(q_k,s_k)\geqslant\sum_{c\in E(k)}\frac2{n_k}\lfloor \frac{Card(c)}3\rfloor>\frac{2\sum_{c\in E(k)}Card(c)}{3n_k}-\frac{e(k)}3.\]
We can estimate the limit $\lim_{k\to\omega}d_H(q_k,s_k)\geqslant 2/3\cdot cyc_\infty(p)>0$. So $q\in\Pi_{k\to\omega}S_{n_k}$ cannot be represented by permutations exactly commuting with $p_k$.

\textbf{Reverse implication.}
Again $p=\Pi_{k\to\omega}p_k$ and we assume now that $cyc_\infty(p)=0$. It follows that $\sum_{i\in\nz^*}cyc_i(p)=1$. Let $q\in\Pi_{k\to\omega}S_{n_k}$ be such that $pq=qp$ and choose $q_k\in S_{n_k}$ such that $q=\Pi_{k\to\omega}q_k$.

Fix $\ve>0$ and let $i\in\nz$ be such that $\sum_{j\leqslant i}cyc_j(p)<1-\ve/3$. There exists $F\in\omega$ such that for any $k\in F$:
\[\sum_{j\leqslant i}cyc_j(p_k)<1-\ve/2\mbox{ and } d_H(p_kq_k,q_kp_k)<\ve/2i.\]
Fix such a $k\in F$. Let $c_1,\ldots, c_t$ be the cycles of $p_k$ of length less or equal to $i$. By the above condition, at most $\ve/2\cdot n_k$ points are outside these cycles. We call a cycle $c_s$ \emph{good} if $p_kq_k(x)=q_kp_k(x)$ for any $x\in c_s$, and \emph{bad} otherwise. As $d_H(p_kq_k,q_kp_k)<\ve/2i$, there are at most $\ve/i\cdot n_k/2$ bad cycles in $p_k$.

Let $c_s$ be a good cycle. Then $q_k(c_s)$ is another cycle in $p_k$ of equal length. This defines a partial map $\vp:\{1,\ldots, t\}\to\{1,\ldots,t\}$, $q_k(c_s)=c_{\vp(s)}$. We can extend this to $\bar\vp:\{1,\ldots, t\}\to\{1,\ldots,t\}$, a total map such that $Card(c_s)=Card(c_{\bar\vp(s)})$.

We now construct $\bar q_k$ as follows: on a good cycle $c_s$ we let $\bar q_k=q_k$; on a bad cycle $c_s$ we choose a bijection such that $\bar q_k(c_s)=c_{\bar\vp(s)}$ and $\bar q_k$ commutes with $p_k$ on $c_s$; outside of cycles $c_1,\ldots,c_t$ we let $\bar q_k$ be the identity. By construction $\bar q_kp_k=p_k\bar q_k$ and $d_H(q_k,\bar q_k)\leqslant [(\ve/i\cdot n_k/2)i+\ve/2\cdot n_k]/n_k=\ve$.
\end{proof}

\end{document}